\documentclass{amsart}
\usepackage{mathtext}
\usepackage{amsmath}
\usepackage{amsfonts}
\usepackage{amssymb}
\usepackage{mathrsfs}
\usepackage{amsthm}
\usepackage{dsfont}
\usepackage{graphicx}
\usepackage{epigraph}
\usepackage{hyperref}

\theoremstyle{definition}\newtheorem{Def}{Definition}
\theoremstyle{plain}\newtheorem{Th}{Theorem}
\theoremstyle{plain}

\theoremstyle{remark}\newtheorem{Rem}{Fact}
\theoremstyle{plain}\newtheorem{Le}{Lemma}
\theoremstyle{plain}
\theoremstyle{plain}\newtheorem{Prop}{Proposition}

\newcommand{\Cii}[1]{_{{}_{\scriptstyle #1}}}
\newcommand{\cii}[1]{_{{}_{#1}}}
\newcommand{\BMO}{\mathrm{BMO}}
\newcommand{\df}{\buildrel\mathrm{def}\over=}

\newcommand{\Hold}{\dot{\mathcal{C}}}
\newcommand{\R}{\mathbb R}
\newcommand{\Z}{\mathbb Z}
\newcommand{\I}{\mathcal I}

\newcommand{\M}[3]{M_{#1,\,#2}^{#3}\,}
\DeclareMathOperator{\supp}{supp}

\hypersetup{pdfstartview = FitH}

\author{Eugenia Malinnikova}
\author{Nikolay N. Osipov}

\thanks{Nikolay N. Osipov is supported by ERCIM ``Alain Bensoussan'' Fellowship Programme and by RFBR (grant no. 14-01-31163 and no. 14-01-00198).}
\address{St.~Petersburg Department of Steklov Mathematical Institute RAS, Fontanka 27, St.~Petersburg, Russia}
\email{nicknick AT pdmi DOT ras DOT ru}

\address{Norwegian University of Science and Technology (NTNU), Department of Mathematical Sciences, N-7491, Trondheim, Norway}
\email{eugenia.malinnikova AT ntnu DOT no}

\title[Two types of Rubio de Francia operators]{Two types of Rubio de Francia operators on Triebel--Lizorkin and Besov spaces}


\begin{document}

\begin{abstract}
	We discuss generalizations of Rubio de Francia's inequality for
Trie\-bel--Li\-zor\-kin and Besov spaces, continuing the research from
\cite{Os}. Two versions of  Rubio de Francia's
operator are discussed:  
it is shown that a rotation factor is needed for the  boundedness of the
operator in some smooth spaces while it is not essential in other spaces.
We study the operators on some ``end'' spaces of the Trie\-bel--Li\-zor\-kin scale
and then use usual
interpolation methods.
\end{abstract}

\maketitle

\section{Introduction}
If $f$ is a function in $L^2([0,1])$ and $I$ is an interval in $\mathbb{Z}$, then 
by $(M_{I}f)(x)$ we denote the exponential polynomial 
$
	(\widehat f \mathds{1}_{I})^{\vee}(x) = \sum_{n\in I} \widehat f(n) e^{2\pi i nx}.
$
For any collection $\mathcal{I}$ of pairwise disjoint intervals $I \subset \mathbb{Z}$ such that 
$\bigcup_{I\in\mathcal{I}} I = \mathbb{Z}$, we have
\begin{equation}\label{PId}
	\|f\|_{L^2} =\Big\|\Big(\sum_{I\in\mathcal{I}} |M_{I}f|^2 \Big)^{1/2}\Big\|_{L^2}.
\end{equation}
This is an equivalent reformulation\footnote{We are talking about the equivalence of two correct 
statements in the sense that they are direct consequences of each other.} of Parseval's 
identity, one of the most fundamental results in harmonic analysis.
For brevity, we can write the right expression in~\eqref{PId} as $\big\|\{M_{I}f\}\cii{I\in\mathcal{I}}\big\|_{L^2(l^2)}$.

In form~\eqref{PId}, Parseval's identity has an extension to the spaces $L^p([0,1])$. Namely, for $2 \le p < \infty$, we have the following
two-sided inequality:
\begin{equation}\label{ClsInq}
	c_p\big\|\{M_{I}f\}\cii{I\in\mathcal{I}}\big\|_{L^p(l^2)} \le \|f\|_{L^p} \le C_p \big\|\{M_{J}f\}\cii{J\in\mathcal{J}}\big\|_{L^p(l^2)},
\end{equation}
where $\mathcal{I}$ is an arbitrary collection of pairwise disjoint intervals in $\mathbb{Z}$, the collection $\mathcal{J}$ is defined as
\begin{equation*}
	\mathcal{J} \df \big\{(-2^{k+1},-2^k]\big\}_{k\in\mathbb{Z}^+}\cup\big\{\{0\}\big\}\cup\big\{[2^k,2^{k+1})\big\}_{k\in\mathbb{Z}^+}\,,
\end{equation*} 
and the constants $c_p$ and $C_p$ depend only on $p$ (in particular, $c_p$ does not depend on the choice of $\mathcal{I}$). The left inequality have been obtained by Rubio de Francia~\cite{Ru} in 1983, and the right inequality is the classical Littlewood--Paley 
theorem (see, e.g., the exposition in~\cite{St}). 
By duality, if we interchange the left and the right expressions in~\eqref{ClsInq}, we obtain correct estimates 
for $1 < p \le 2$, provided $\bigcup_{I\in\mathcal{I}} I = \mathbb{Z}$. 

In what follows, we consider the whole line~$\mathbb{R}$ instead of~$[0,1]$. 
In such a context, the Fourier transform $\widehat f$ is also defined on~$\mathbb{R}$ 
(so we consider collections of intervals on~$\mathbb{R}$) and
relation~\eqref{ClsInq} remains true, provided $k$ runs over the whole~$\mathbb{Z}$ in the definition of~$\mathcal{J}$. 
In fact, the corresponding results are usually presented precisely in this form (see~\cite{Ru, St}).

Next, we note that $L^p$-classes do not exhaust the set of spaces studied in harmonic analysis. 
In addition to them, there are many normed spaces that seem, at first glance, to have no direct connection with each other:
Sobolev spaces, the $\BMO$-space, H\"older--Zygmund classes of smooth functions, etc. 
But it is known that the corresponding norms can be written in a uniform way: 
all these spaces belong to the scale of Triebel--Lizorkin and Besov spaces. 
In this article, we outline an overall picture: we discuss generalizations of Rubio de Francia's inequality for a substantial part of Besov--Triebel--Lizorkin scale (which includes all of the spaces listed above).
In this general context, we raise and answer a subtle question concerning the presence or absence of the rotations in the operators that correspond to Rubio de Francia's inequality.

Now, let $\mathcal{I} = \{I_m\} = \big\{[a_m,b_m]\big\}$ be a finite or countable collection
of pairwise disjoint intervals in~$\mathbb{R}$ such that 
\begin{equation}\label{ZeroNotContained}
	0 \notin (a_m,b_m)
\end{equation} 
for any~$m$. 
Suppose $\varphi$ is a Schwartz function such that 
$
	\supp \widehat\varphi \subset (0,1)
$
(in particular, $\supp \widehat\varphi$ is separated from~$0$ and~$1$).
We introduce the functions~$\varphi_m$ corresponding to the intervals~$I_m$:
\begin{equation}\label{phim}
	\widehat\varphi_m(t) = \widehat\varphi\bigg(\frac{t - a_m}{b_m - a_m}\bigg).
\end{equation}
Consider two operators that transform scalar-valued functions to collections of functions by the following formulas:
\begin{equation}\label{Oper}
	S\cii{\mathcal{I}}^{\varphi}f\,(x) \df \big\{(f\ast\varphi_m)(x)\big\}_m\quad\mbox{and}\quad \widetilde{S}\cii{\mathcal{I}}^{\varphi}f\,(x) \df \big\{e^{-2\pi i\, a_m x}(f\ast\varphi_m)(x)\big\}_m.
\end{equation}
Also we introduce two corresponding families of operators
$$
	\mathfrak{S}_{\varphi} \df \big\{S\cii{\mathcal{I}}^{\varphi}\big\}_{\mathcal{I}}\quad\mbox{and}\quad 
	\widetilde{\mathfrak{S}}_{\varphi} \df \big\{\widetilde{S}\cii{\mathcal{I}}^{\varphi}\big\}_{\mathcal{I}},
$$
where $\mathcal{I}$ runs over all possible collections of pairwise disjoint intervals in~$\mathbb{R}$ 
satisfying~\eqref{ZeroNotContained}.

The fact that for $2 \le p < \infty$  
the family~${\mathfrak{S}}_{\varphi}$ is uniformly bounded from $L^p$ to $L^p(l^2)$ is a version of Rubio de Francia's
theorem where we have substituted smooth multipliers~$\varphi_m$ instead of $\mathds{1}_{I_m}$.\footnote{In the original form his result cannot
 be extended to some of the Besov and Triebel--Lizorkin spaces (see~\cite{Os}). 
In this article, we do not want to touch on issues that arise when dealing with non-smooth multipliers.} 
Its proof is contained in considerations of~\cite{Ru}. In fact, Rubio de Francia deals with the family~$\widetilde{\mathfrak{S}}_{\varphi}$. The matter is that the factors $e^{-2\pi i\, a_m x}$ played a significant role in the proof: their presence allows to get a Calder\'on--Zygmund type condition for the kernels of~$\widetilde{S}\cii{\mathcal{I}}^{\varphi}$. 
But since the $L^p$-norms are invariant under multiplications by unimodular functions and, in particular, are rotation-invariant, the exponential functions can be dropped. 
Now we note that the norms in all the other Trie\-bel--Li\-zor\-kin spaces as well as in the Besov spaces are not rotation-invariant. 
Therefore the boundedness of the families~$\mathfrak{S}_{\varphi}$ and~$\widetilde{\mathfrak{S}}_{\varphi}$ should be studied separately on such spaces. 

Some studies concerning the family with rotations can be found in~\cite{Os}, where the author considers pointwise estimates 
for the operators~$\widetilde{S}\cii{\mathcal{I}}^{\varphi}$ in terms of sharp (oscillatory) maximal functions. 
In particular, the results of~\cite{Os} imply that $\widetilde{\mathfrak{S}}_{\varphi}$ is uniformly bounded on the H\"older--Zygmund spaces $\Hold_s$ as well as on $\BMO$. But it turns out that in 
the context of the Besov--Triebel--Lizorkin scale those pointwise estimates give much more: we are going to rely heavily on them in our considerations below.

The family~$\mathfrak{S}_{\varphi}$ is also studied below. In particular, we are going to show that it is not bounded on~$\Hold_s$ or~$\BMO$. But surprisingly, it turns out that the both of our families 
are uniformly bounded on some other Trie\-bel--Li\-zor\-kin and Besov spaces with the norms that are not rotation-invariant.

\section{Preliminaries}
\subsection{Triebel--Lizorkin and Besov spaces.} We restrict ourselves to considering only functions on the real line~$\mathbb{R}$.
Let $\mathcal{S}$, $\mathcal{S}'$, and $\mathcal{P}$ be Schwartz space,  the space of tempered distributions, and the space of all 
algebraic polynomials respectively.

Consider a function $\phi \in \mathcal{S}$ such that $\supp \widehat \phi \subset [-2, 2]$ and $\widehat \phi \equiv 1$ on $[-1,1]$. If we introduce functions $\phi_j$ by the formula
\begin{equation}\label{unity}
	\widehat{\phi}_j(\xi) = \widehat{\phi}\big(2^{-j}\xi\big) - \widehat{\phi}\big(2^{-j+1}\xi\big),\quad j \in \mathbb{Z},
\end{equation}
then the collection $\{\phi_j\}_{j \in \mathbb{Z}}$ will be \emph{a resolution of unity}, i.e., we will have
\begin{equation}\label{phijsupp}
 \supp\widehat{\phi}_j \subset \big[-2^{j+1},\, -2^{j-1}\big] \cup \big[2^{j-1},\, 2^{j+1}\big].
\end{equation}
and
$$
\sum_{j \in \mathbb{Z}} \,\widehat{\phi}_j \equiv 1 \quad\mbox{on}\quad \mathbb{R}\setminus \{0\}.
$$

\begin{Def}\label{TribLiz}
	Let $0<p<\infty$, $0<q\le \infty$, and $s\in \mathbb{R}$. 
	We say
	that an element~$f$ of the quotient space $\mathcal{S}'/\,\mathcal{P}$ belongs to the homogeneous Triebel--Lizorkin space~$\dot{F}_{pq}^{s}$ if
	$$
		\|f\|_{\dot{F}_{pq}^{s}} \df \Big\|\,\Big|\big\{2^{js} f\ast\phi_j\big\}_{j \in \mathbb{Z}}\Big|_{l^q}\Big\|_{L^p} < \infty.
	$$
\end{Def}
If we permute the $L^p$- and $l^q$-norms, we obtain a definition of the Besov spaces~$\dot{B}_{pq}^{s}$.
\begin{Def}\label{Besov}
	Let $0<p\le\infty$, $0<q\le \infty$, and $s\in \mathbb{R}$. 
	We say that $f \in \mathcal{S}'/\,\mathcal{P}$ belongs to the homogeneous Besov space~$\dot{B}_{pq}^{s}$ if
	$$
		\|f\|_{\dot{B}_{pq}^{s}} \df \Big|\Big\{\big\|2^{js} f\ast\phi_j\big\|_{L^p}\Big\}_{j \in \mathbb{Z}}\,\Big|_{l^q} < \infty.
	$$
\end{Def}
Note that we have not define the spaces~$\dot{F}_{\infty q}^{s}$. It turns out that a direct extension of 
Definition~\ref{TribLiz} to $p=\infty$ is not reasonable. Such a space would depend on the choice of a dyadic resolution of unity participating in the definition (see~\cite{Trib2}). A correct definition of~$\dot{F}_{\infty q}^{s}$ follows from duality arguments and can be found, e.g., in~\cite{FrJa, Trib2, Trib}.

There are some well-known facts about Triebel--Lizorkin and Besov spaces.
\begin{Prop}\label{FIs} We have
\begin{enumerate}
	\item[(i)] $\dot{F}_{pp}^{s} = \dot{B}_{pp}^{s} \;\;\mbox{if}\;\; 0<p<\infty$\textup;\vskip5pt
	\item[(ii)] $\dot{B}_{\infty\infty}^{s} \cong \Hold^s \;\;\mbox{if}\;\; s > 0$\textup;\vskip5pt
	\item[(iii)] $\dot{F}_{p2}^{k} \cong \dot{W}^p_k \;\;\mbox{if}\;\; 1<p<\infty \;\;\mbox{and}\;\; k\in\mathbb{Z}_+$\textup;\vskip5pt
	\item[(iv)] $\dot{F}_{\infty 2}^{0} \cong \BMO$.
\end{enumerate}
\end{Prop}
Here by $\Hold^s$, $s >0$, we denote the homogeneous H\"older--Zygmund spaces. The corresponding definition
can be found, e.g., in~\cite[1.4.5]{Trib}. In the same place the Besov norm is presented in the form that immediately implies (ii).
Here we only note that if $s \notin\mathbb{Z}_+$, then the norm in $\Hold^s$  is equivalent to the corresponding H\"older norm:
$$
	\|f\|_{\Hold^s} \cong \sup_{x\ne y}\frac{\big|f^{(k)}(x) - f^{(k)}(y)\big|}{|x-y|^{s-k}},\quad k = [s].
$$
Concerning (iii) and (iv), see~\cite[Chapter~5]{Trib4}. Here $\dot{W}^p_k$ are homogeneous Sobolev spaces, and (iii) includes, 
in particular, the fact that $\dot{F}_{p2}^{0} = L^p$, $1<p<\infty$.

\subsection{Sharp maximal functions.} Let $\mathcal{P}_i$ be the space of algebraic polynomials of degree strictly less than~$i$. We agree that ${\mathcal{P}_0 = \{0\}}$.
\begin{Def}\label{DefOfM}
Suppose\footnote{A wider range of parameters $p$ and $s$ can be considered in this context, but those that are indicated here suffice for our goals.} $1\le p< \infty$, $i\in\mathbb{Z}_+$, and $s\in[0,i]$. 
Let $h$ be a measurable function on $\mathbb{R}$.
We define the maximal function $\M{i}{s}{p} h$ by the formula
$$
	\M{i}{s}{p} h(x)\df
	\sup_{I \ni x}\inf_{P}\frac{1}{|I|^{s}}\bigg(\frac{1}{|I|}\int\limits_{I}|h-P|^p\bigg)^{1/p},
$$
where the supremum is taken over all the intervals containing~$x$ and the infimum is taken over all the polynomials $P \in \mathcal{P}_i$.
\end{Def}
\begin{Def}\label{DefOfCps}
	Let ${1\le p<\infty}$ and $s>0$. Suppose $f \in L_{\mathrm{loc}}^1/\,\mathcal{P}_{[s]+1}$.
	We say that $f \in \Hold_p^s$ if
	$$
		\|f\|_{\Hold_p^s} \df \big\|\M{[s]+1}{s}{p} f\big\|_{L^p} < \infty.
	$$
\end{Def}
We can extend this definition to $\Hold_\infty^s$. It is known (see~\cite{Ca, Me} and the exposition in~\cite{KiKr}) that 
the quantities $\|\M{i}{s}{p} f\|\cii{L^\infty}$ are equivalent for various $p$, and so we put
\begin{equation}\label{Csinfty}
	\|f\|_{\Hold_\infty^s} \df \big\|\M{[s]+1}{s}{2} f\big\|_{L^\infty}.
\end{equation}
We have (see~\cite{Ca, KiKr, Me})
$$
	\Hold_\infty^s \cong \Hold^s.
$$  

Following Triebel~\cite[1.7.2]{Trib}, we put 
\begin{equation}\label{FInfty}
	\dot{F}_{\infty \infty}^{s} \df \dot{B}_{\infty \infty}^{s} \cong \Hold^s,\quad s>0, 
\end{equation}	
and state the following fact.
\begin{Prop}\label{FIsC}
	If $1\le p\le\infty$ and $s>0$, then for $f\in \dot{F}_{p\infty}^{s}$ we have 
	\begin{equation}\label{FViaM}
		\dot{F}_{p\infty}^{s} \cong \Hold_p^s.
	\end{equation}
\end{Prop}
This proposition is a consequence of~\cite[Theorem~1]{Se}.

\subsection{Interpolation.}
The interpolation between Triebel--Lizorkin spaces is one of the main components of our subsequent considerations. 
\begin{Prop}
Interpolating between $\dot{F}_{pq}^{s}$-spaces, we can obtain another Tri\-eb\-el--Li\-zor\-kin space as well as a Besov space depending on the interpolation method we use.
\begin{enumerate}
	\item[(i)]
		Let $s_0,s_1 \in \mathbb{R}$, $1 \le q_0 < \infty$, $1 \le q_1 \le \infty$, and $1\le p_0,p_1 < \infty$. 
		Suppose
		\begin{gather*}
		0 <\theta< 1,\quad s= (1-\theta)s_0 + \theta s_1,\rule{0pt}{12pt}\\ 
		\frac{1}{p} = \frac{1-\theta}{p_0} + \frac{\theta}{p_1},\quad\mbox{and}\quad \frac{1}{q} = \frac{1-\theta}{q_0} + \frac{\theta}{q_1}. \rule{0pt}{18pt}
		\end{gather*}
		Applying the complex interpolation method, we have
		\begin{equation}\label{CompInterp}
			\big [\dot{F}_{p_0q_0}^{s_0},\, \dot{F}_{p_1q_1}^{s_1}\big ]_{\theta} = \dot{F}_{pq}^{s}.
		\end{equation}
	\item[(ii)]
		Let $s_0,s_1 \in \mathbb{R}$, $s_0 \ne s_1$, $0 < q_0,q,q_1 \le \infty$, and $0<p<\infty$. As above, suppose
		$$
			0 <\theta< 1\quad\mbox{and}\quad s= (1-\theta)s_0 + \theta s_1.
		$$
		Applying the real interpolation method, we have
		\begin{equation}\label{RealInterp}
			\big (\dot{F}_{pq_0}^{s_0},\, \dot{F}_{pq_1}^{s_1}\big )_{\theta, q} = \dot{B}_{pq}^{s}.
		\end{equation}
\end{enumerate}
\end{Prop}
Part~(i) of this theorem is contained in~\cite[Corollary 8.3]{FrJa}. Here 
$[\cdot,\cdot]_\theta$ is the classical complex interpolation method with the interpolation property.
Concerning part~(ii), see~\cite[2.4.2, 5.2.5]{Trib4}

\subsection{Vector-valued spaces.}
Let $X$ be a Triebel--Lizorkin ($p\ne\infty$) or Besov space. Then by $X_*$ we denote the space of sequences
\begin{equation*}
	f = \{f_m\}\cii{m\in\mathbb{N}},\quad f_m\in \mathcal{S}'/\,\mathcal{P},
\end{equation*}
equipped with the corresponding norm where we substitute lengths in~$l^2$ instead of absolute values. 
For example, if $X = \dot{F}_{pq}^{s}$, then $X_*$ has the norm
$$
	\|f\|\cii{X_*}
	=\Big\|\,\Big|\Big\{\big|\big\{2^{js} f_m\ast\phi_j\big\}_{m\in\mathbb{N}}\big|_{l^2}\Big\}_j\Big|_{l^q}\Big\|_{L^p}.
$$
We leave the reader to determine what will be the norm in $X_*$ if we put $X = \dot{B}_{pq}^{s}$.
By $X_N$, $N\in\mathbb{N}$, we denote the subspace in~$X_*$ consisting of sequences 
such that $f_m = 0$ for $m > N$. 
Similarly substituting $l^2$-norms instead of absolute values, we can also introduce the maximal functions~$\M{i}{s}{p}$ as well as the spaces~$\Hold_p^s$ (see Definitions~\ref{DefOfM} and~\ref{DefOfCps}) for finite or countable collections of functions.

Since there is no difference whether we deal with absolute values or with lengths of finite-dimensional vectors,
we can assert the following. 
\begin{Rem}\label{N1}
	All aforecited facts on Triebel--Lizorkin or Besov spaces~$X$ remain true for the corresponding spaces~$X_N$ 
	independently on~$N$.
\end{Rem}

Next, since the $l^2$-norm is a limit of an increasing non-negative sequence, we have
$$
	\big\|\{f_m\}\cii{m\in\mathbb{N}}\big\|_{X_*} = \lim_{N\to \infty}\big\|\{f_1,\dots,f_N, 0,\dots\}\big\|_{X_N}
$$
and, therefore, it suffice to deal only with the spaces~$X_N$. Namely, we can state the following fact.
\begin{Rem}\label{N2}
	If for finite collections~$\mathcal{I} = \{I_m\}_{m=1}^N$ 
	of intervals the operators $S\cii{\mathcal{I}}^{\varphi}$ and 
	$\widetilde{S}\cii{\mathcal{I}}^{\varphi}$ are bounded from 
	$X$ to $X_N$  uniformly in~$N$ and~$\mathcal{I}$, then 
	this remains true for countable collections~$\mathcal{I}$: 
	the families~$\widetilde{\mathfrak{S}}_{\varphi}$ and	${\mathfrak{S}}_{\varphi}$
	are uniformly bounded from $X$ to~$X_*$.
\end{Rem}

Using considerations from~\cite{KiKr,Os}, we can prove the following 
proposition. 
\begin{Prop}
Suppose $2 \le p < \infty$, $i \in \mathbb{Z}_+$, and $s \in [0,i)$. 
If $f$ is a measurable function such that $\M{i}{s}{p}f$ is finite at least at one point, then $f \in \mathcal{S}'$ and
we have the following pointwise estimate:
\begin{equation}\label{PointEst}
	\M{i}{s}{p}\big(\widetilde{S}\cii{\mathcal{I}}^{\varphi}f\big) \le C \M{i}{s}{p}f,
\end{equation}
where the constant~$C$ does not depend on~$\mathcal{I}$ or~$f$.
\end{Prop}
In~\cite{Os}, a similar estimate is proved for $p = 2$ and for non-smooth multipliers $\widehat\varphi_m$. 
But Rubio de Francia's~\cite{Ru} theorem allows to prove that 
the same method can be employed for all $p \ge 2$; and the smoothness of $\widehat\varphi_m$ simplifies the arguments. 
Also we note that this is the very place where we need the set $\supp \widehat\varphi$ to be separated from~$0$ and~$1$.

Relations~\eqref{FViaM} and~\eqref{PointEst} together with Facts~\ref{N1} and~\ref{N2} imply the following consequence.
\begin{Prop}\label{LowerBoundary}
	Let $2\le p\le\infty$ and $s>0$. If we put $X = \dot{F}_{p\infty}^{s}$, 
	then the family~$\widetilde{\mathfrak{S}}_{\varphi}$ will be uniformly bounded from~$X$ to~$X_*$. 
\end{Prop}
We also have (see~\cite{Os} again) the following proposition.
\begin{Prop}\label{BMO}
	Let $X = \dot{F}_{\infty 2}^{0} \cong \BMO$. 
	Then the family~$\widetilde{\mathfrak{S}}_{\varphi}$ will be uniformly bounded from~$X$ to~$X_* \df \BMO(l^2)$.
\end{Prop}

\section{Formulation of the results}
\begin{Def}\label{nondeg}
We say that~$\gamma \in L^1$ is non-degenerate if
$
	\int_{-\infty}^{0} e^{2\pi it}\gamma(t)\,dt \neq 0.
$
\end{Def}
The following fact justifies the term ``non-degenerate''.
\begin{Rem} 
	If~$\gamma\in L^1$ is a non-zero function such that $\widehat\gamma$ is non-negative and supported in $[0,1]$, then $\gamma$ is non-degenerate.
\end{Rem}
\begin{proof}
	Let $\Phi_\gamma(x) \df \int_{-\infty}^xe^{2\pi it}\gamma(t)\,dt$. We have $\widehat{\Phi}_\gamma(\xi)=(2\pi i\xi)^{-1}\,\widehat{\gamma}(\xi-1)$ and 
\[
\int\limits_{-\infty}^0e^{2\pi i t}\gamma(t)\,dt = \Phi_\gamma(0)=\int\limits_{\R}\widehat{\Phi}_\gamma(\xi)\,d\xi=\frac{1}{2\pi i}\int\limits_{1}^2\frac{\widehat{\gamma}(\xi-1)}{\xi}\,d\xi.\]
Since $\widehat{\gamma}\ge 0$ and does not vanish on $[0,1]$, we get $\Phi_\gamma(0)\neq 0$ and $\gamma$ is non-degenerate.
\end{proof}
Now we are ready to present our results.
\begin{Th}\label{ThTrib}
	Let $X = \dot{F}^{s}_{pq}$. We determine various ranges for $p$, $q$, and $s$ for each case considered below.
	\begin{enumerate}
	\item[(i)] Let $2 \le p \le \infty$, $2 \le q \le \infty$, and $s > 0$. We modify this domain as follows (see also Figure~\textup{\ref{interp}}):
		\begin{itemize}
			\item if $q=2$, then for $p \ne \infty$ we consider all $s \ge 0$;
			\item if $q=2$ and $p= \infty$, then we consider only $s=0$;
			\item if $p = \infty$, then we exclude $q \in (2,\infty)$ from consideration.
		\end{itemize}
		If $p$, $q$, and $s$ belong to the domain just described, then the family~$\widetilde{\mathfrak{S}}_{\varphi}$ is uniformly bounded 
		from~$X$ to~$X_*$.
	\item[(ii)]
		For $2 \le q \le p < \infty$ and $s \ge 0$ (see Figure~\textup{\ref{InterpSComplex}}),
		the family~${\mathfrak{S}}_{\varphi}$ is uniformly bounded from~$X$ to~$X_*$.
	\item[(iii)]
		If $X = \dot{F}^{0}_{\infty 2}$ or $X = \dot{F}^{s}_{\infty \infty}$, $s>0$, then there exists a collection~$\mathcal{I}$ 
		of pairwise disjoint intervals 
		such that the operator $S\cii{\mathcal{I}}^{\varphi}$ is \emph{not} bounded from~$X$ to~$X_*$ provided $\varphi$ is non-degenerate.
	\end{enumerate}
\end{Th}
So there are Triebel--Lizorkin spaces where only the family~$\widetilde{\mathfrak{S}}_{\varphi}$ 
is uniformly bounded as well as spaces where both families~${\mathfrak{S}}_{\varphi}$ 
and~$\widetilde{\mathfrak{S}}_{\varphi}$ are uniformly bounded (in spite of the fact that the corresponding norms are not ro\-ta\-tion-in\-va\-ri\-ant).

Similar result holds for the Besov spaces. 
Namely, we have the following theorem.
\begin{Th}\label{ThBesov}
	Let $X = \dot{B}^{s}_{pq}$.
	\begin{enumerate}
		\item[(i)] Let $2 \le p \le \infty$, $0<q \le \infty$, and $s > 0$. 
		If $p = \infty$, then we exclude all $q \ne \infty$ from consideration (see Figure~\textup{\ref{InterpTildeSReal}}). For such $p$, $q$, and $s$, the family~$\widetilde{\mathfrak{S}}_{\varphi}$ is uniformly bounded 
		from~$X$ to~$X_*$.
		\item[(ii)] For $2 \le p < \infty$, ${0<q \le \infty}$, and ${s \ge 0}$, the family~${\mathfrak{S}}_{\varphi}$ is 
		uniformly bounded from~$X$ to~$X_*$.
		\item[(iii)] Let $X = \dot{B}^{s}_{\infty q}$ for ${0<q \le \infty}$ and ${s \ge 0}$. Then there exists a collection~$\mathcal{I}$ 
		of pairwise disjoint intervals such that the operator $S\cii{\mathcal{I}}^{\varphi}$ is \emph{not} bounded from~$X$ to~$X_*$ provided 
		$\varphi$ is non-degenerate.
	\end{enumerate}
\end{Th}
As we will see, there is a deep connection between Theorems~\ref{ThTrib} and~\ref{ThBesov}. The point is that in order to prove their first parts, we will, in fact, interpolate between the same spaces, but applying two different methods of interpolation. 

We also mention the following non-linear quadratic operator that transform scalar-valued functions to scalar-valued functions:
$$
	G\cii{\mathcal{I}}^{\varphi} f = \big|S\cii{\mathcal{I}}^{\varphi}f\big|_{l^2} = \big|\widetilde{S}\cii{\mathcal{I}}^{\varphi} f\big|_{l^2}
$$
It is a more ``rough'' operator: treating it, we deal with expressions 
of the form $\big||a|_{l^2}-|b|_{l^2}\big|$ instead of $|a-b|_{l^2}$ 
(it becomes clear what we mean if we put, e.g., $X = \Hold^s$). 
We also note that in order to study the operator~$G\cii{\mathcal{I}}^{\varphi}$, we should answer the question: 
can the sequences $S\cii{\mathcal{I}}^{\varphi}f$ or $\widetilde{S}\cii{\mathcal{I}}^{\varphi} f$ be presented 
as $l^2$-valued functions? We assume this neither in the definition of vector-spaced spaces~$X_*$ nor elsewhere above. 
Concerning this question, see, e.g., \cite[Fact~2.1]{Os}. Here we do not investigate problems related to the operator 
$G\cii{\mathcal{I}}^{\varphi}$ anymore.

\section{The proofs}
In order to prove parts~(i) and~(ii) of Theorems~\ref{ThTrib} and~\ref{ThBesov}, 
it suffices (due to Fact~\ref{N2}) to consider finite collections~$\mathcal{I}$ of intervals that determine the operators. In this case,
Fact~\ref{N1} allows to employ the whole theory of Triebel--Lizorkin and Besov spaces.

\subsection{Proof of Theorem~\ref{ThTrib}, part (i)}

	First, we consider the space~$\dot{F}^{k}_{22}$, $k \in \mathbb{Z}_+$, which coincides with~$\dot{W}^{2}_{k}$ (see Proposition~\ref{FIs}). 
	
	\begin{Le}\label{W2}
		If $X=\dot{F}^{k}_{22}$, $k\in\mathbb{Z}_+$, then the family~$\widetilde{\mathfrak{S}}_{\varphi}$ is uniformly bounded from~$X$ to~$X_*$.
	\end{Le}
	\begin{proof}
	Let $f \in X$. We put 
	$$
		g_m(x) \df e^{-2\pi i\, a_m x}(f\ast\varphi_m)(x)
	$$ 
	(i.e., we have $\{g_m\} = \widetilde{S}\cii{\mathcal{I}}^{\varphi}f$), and
	by the Plancherel theorem, we can write
	\begin{equation}\label{Wl2}
		\big\|\big\{g_m^{(k)}\big\}\big\|_{L^2(l^2)}^2 = \sum_m\int\limits_0^{l_m} \big|\hat f(\xi + a_m)\big|^2\, \big|\widehat\varphi_m(\xi+a_m)\big|^2\, |\xi|^{2k} \,d\xi,
	\end{equation}
	where $l_m = b_m - a_m$.
	Rewrite each term with $a_m \ge 0$ as
	$$
		\int\limits_{I_m} \big|\hat f(\xi)\big|^2\, \big|\widehat\varphi_m(\xi)\big|^2\, |\xi-a_m|^{2k} \,d\xi.		
	$$
	In this case we have $|\xi - a_m| \le |\xi|$, and, therefore, it can be estimated by 
	$$
		C\!\int\limits_{I_m} \big|\hat f(\xi)\big|^2\, |\xi|^{2k} \,d\xi.
	$$
	
	For all the remaining terms in~\eqref{Wl2}, we have $b_m \le 0$, because $0 \notin (a_m,b_m)$ for all~$m$.
	In this case, 
	we rewrite the discussed terms as
	\begin{equation}\label{bterms}
		\int\limits_{-l_m}^0 \big|\hat f(\xi + b_m)\big|^2\, \big|\widehat\varphi_m(\xi+b_m)\big|^2\, |\xi+l_m|^{2k} \,d\xi
	\end{equation}
	and get rid of $l_m$ in the last factor.\footnote{Thus, we show that it is not significant whether we shift~$a_m$ or~$b_m$ to the origin.} For this we verify that for $\xi \in (-l_m, 0)$, we have
	\begin{equation}\label{atob}
		\big|\widehat\varphi_m(\xi+b_m)\big|^2\, |\xi+l_m|^{2k} \le C_k|\xi|^{2k},
	\end{equation}
	where $C_k$ does not depend on~$m$.
	We have
	$
	\widehat\varphi_m(\xi+b_m)	= \widehat\varphi(1+\xi/l_m),
	$
	and in order to prove~\eqref{atob}, we only need to verify that 
	$$
		\widehat\varphi(1-t) \le C_k\bigg(\frac{t}{1-t}\bigg)^k \quad\mbox{for}\quad t \in (0,1).
	$$
	But this is true because $\widehat\varphi$ equals zero at $1$ with all its derivatives.\footnote{We could also use the fact 
	that $\supp\widehat\varphi$ is separated from~$0$ and~$1$, but we do not need this restriction in order to prove Lemma~\ref{W2}.} 
	Thus, we have that~\eqref{bterms} can be estimated by
	$$
		C\!\!\int\limits_{-l_m}^0 \big|\hat f(\xi + b_m)\big|^2\, |\xi|^{2k} \,d\xi = C\!\int\limits_{I_m} \big|\hat f(\xi)\big|^2\, |\xi - b_m|^{2k} \,d\xi.
	$$
	But since we consider the terms where~$b_m\le 0$, the last expressions are lesser than
	$$
		C\!\int\limits_{I_m} \big|\hat f(\xi)\big|^2\, |\xi|^{2k} \,d\xi.
	$$
	Combining it all together, at least for finite collections~$\mathcal{I}$ we obtain
	\begin{equation*}
		\big\|\widetilde{S}\cii{\mathcal{I}}^{\varphi}f\big\|_{X_*}^2  \cong
		\big\|\big\{g_m^{(k)}\big\}\big\|_{L^2(l^2)}^2 \le 
		C\sum_m \int\limits_{I_m} \big|\hat f(\xi)\big|^2\, |\xi|^{2k} \,d\xi \le C\|f\|_{X}^2.
	\end{equation*}
	Due to Fact~\ref{N2}, the lemma is proved.
\end{proof}		
		
		We know that 
		$\widetilde{\mathfrak{S}}_{\varphi}$ is bounded on 
	$\dot{F}^{0}_{\infty 2} \cong \BMO$ (see Proposition~\ref{BMO}) and on $\dot{F}^{s}_{\infty \infty} \cong \Hold^s$ as well (see Proposition~\ref{LowerBoundary}). 

Finally, let $2\le p < \infty$.  
We have the boundedness on $\dot{F}^{0}_{p2} \cong L^p$ 
(see Rubio de Francia's~\cite{Ru} theorem), on $\dot{F}^{k}_{22}$ for $k\in\mathbb{Z}_+$
(see Lemma~\ref{W2} just proved), and on $\dot{F}_{p\infty}^{s}$ for $s>0$ (see Proposition~\ref{LowerBoundary}).
Using the complex interpolation method~\eqref{CompInterp} for the couples 
$\big\{\dot{F}^{0}_{p2},\,\dot{F}^{k}_{22}\big\}$ and 
$\big\{\dot{F}^{0}_{p2},\dot{F}_{p\infty}^{s}\}$,
we come to the desired result (see Figure~\ref{interp}).
\begin{figure}
\includegraphics[scale = 0.765]{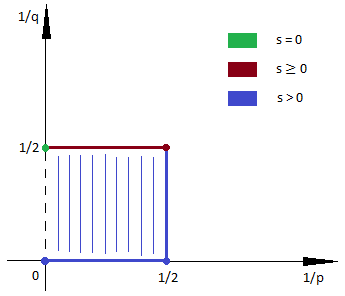}
\caption{Part (i) of Theorem~\ref{ThTrib}.}
\label{interp}
\end{figure}

\subsection{Proof of Theorem~\ref{ThBesov}, part (i)}

We already know that $\widetilde{\mathfrak{S}}_{\varphi}$ is uniformly bounded on $\dot{B}^{s}_{\infty\infty} = \Hold^s$. 
For the remaining spaces we can use the real interpolation method~\eqref{RealInterp}. Indeed, suppose $2 \le p<\infty$ and $s_0,s_1 >0$. Then part~(i) of Theorem~\ref{ThTrib} implies that if we take 
$\big\{\dot{F}^{s_0}_{p\infty},\, \dot{F}^{s_1}_{p\infty}\big\}$ or $\big\{\dot{F}^{s_0}_{p2},\, \dot{F}^{s_1}_{p2}\big\}$ as an interpolation couple, we come to 
the desired result (see Figure~\ref{InterpTildeSReal}).
\begin{figure}
\includegraphics[scale = 0.765]{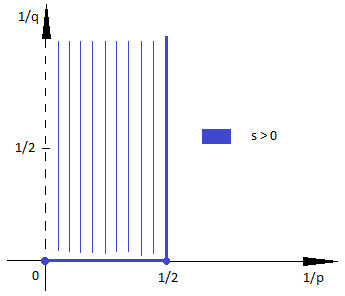}
\caption{Part (i) of Theorem~\ref{ThBesov}.}
\label{InterpTildeSReal}
\end{figure}

\subsection{Proof of Theorem~\ref{ThBesov}, part (ii)}

Denote $g_m \df f\ast\varphi_m$. 
By Rubio de Francia's~\cite{Ru} theorem, we have
$$
	\big\|\{g_m\ast\phi_j\}\Cii{m}\big\|_{L^p(l^2)} = \big\|\{f\ast\phi_j\ast \varphi_m\}\Cii{m}\big\|_{L^p(l^2)} \le C\|f\ast\phi_j\|\Cii{L^p}
$$
and multiplying by  $2^{sj}$ and taking $l^q$ norms we obtain the required boundedness. 
	
\subsection{Proof of Theorem~\ref{ThTrib}, part (ii)}

First, consider the spaces
$$
	X = \dot{F}^k_{p 2} \cong \dot{W}^{p}_k,\quad k \in \mathbb{Z}^+, \quad 2 \le p < \infty. 
$$	
Suppose $f \in X$. By Rubio de Francia's~\cite{Ru} theorem, we have
$$
	\big\|\big\{(f\ast\varphi_m)^{(k)}\big\}\big\|_{L^p(l^2)} = \big\|\big\{f^{(k)}\ast\varphi_m\big\}\big\|_{L^p(l^2)} = \big\|{S}\cii{\mathcal{I}}^{\varphi}\big[f^{(k)}\big]\big\|_{L^p(l^2)} \le C\|f\|_X.
$$
Therefore, in the case being considered, we have the desired result. 

But due to part~(ii) of Theorem~\ref{ThBesov}, we also know that~${\mathfrak{S}}_{\varphi}$ is uniformly bounded on the spaces 
$$
	X = \dot{F}^s_{pp} = \dot{B}^s_{pp},\quad 2 \le p < \infty,\quad s \ge 0.
$$
Using the complex interpolation method~\eqref{CompInterp} (also see Figure~\ref{InterpSComplex}), we conclude the proof.
\begin{figure}
\includegraphics[scale = 0.765]{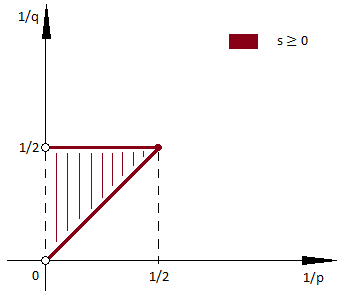}
\caption{Part (ii) of Theorem~\ref{ThTrib}.}
\label{InterpSComplex}
\end{figure}

\subsection{Proof of Theorem~\ref{ThBesov}, part (iii)}
Suppose that $\varphi$ is non-degenerate and set 
$$
\I=\big\{[1+2^{m},\,1+2^{m+1}]\big\}_{m\in\Z_-}.
$$ 

Now consider our functions~$\phi_j$ that are generated by the function~$\phi$ and form a resolution of unity (see \eqref{unity}).
Due to~\eqref{phijsupp} we have 
$$
	\supp \widehat{\phi}_{0} \subset [-2,\,-1/2]\cup [1/2,\,2].
$$
Without loss of generality we can additionally assume that $\widehat \phi \equiv 1$ on $[-3/2,\,3/2]$. Then we also have
\begin{equation}\label{phi0eq1}
\widehat{\phi}_{0} \equiv 1\quad\mbox{on}\quad [-{3}/{2},\, -1] \cup [1,\, {3}/{2}].
\end{equation}

We define
\[f_0 \df \big[\exp(2\pi i \,\cdot) \,\mathds{1}_{[0,+\infty)}(\cdot)\big]\ast{\phi}_{0}.\]  
By~\eqref{phijsupp} the function
$
	f_0 \ast \phi_j
$
does not vanish
only if $j = -1,\,0,\,1$.
This fact, together with Definition~\ref{Besov} and the obvious estimate $|f_0 \ast \phi_j| \le C\|f_0\|\Cii{L^\infty}$, implies
$$
	\|f_0\|\Cii{X} \le C\,\|f_0\|\Cii{L^\infty} < \infty.
$$

Next, we set 
$$
	g_{m} \df f_0\ast\varphi_{m} \quad\mbox{and}\quad g_{m}^j \df g_m\ast \phi_j,
$$ 
where $\varphi_{m}$ is the function that corresponds (see~\eqref{phim}) to the interval $[1+2^m, 1+2^{m+1}]$: 
\[\varphi_{m}(x)=\exp\big(2\pi i\, (1+2^m)x\big)\,2^m\varphi(2^m x).\]
If $2^{m+1}\le 1/2$, then due to~\eqref{phi0eq1} we have
\begin{align*}
	g_{m}^{0} &= \big[\exp(2\pi i \,\cdot) \,\mathds{1}_{[0,+\infty)}(\cdot)\big]\ast\phi_0 \ast \phi_{0} \ast \varphi_{m} \\
	&= \big[\exp(2\pi i \,\cdot) \,\mathds{1}_{[0,+\infty)}(\cdot)\big]\ast\varphi_{m} = g_m.
\end{align*}
Therefore, in this case we can write
\begin{equation}\label{eq:g}
\begin{aligned}
	g_{m}^{0}(x) = g_m(x)
	&= e^{2\pi i x}\int\limits_{-\infty}^{2^mx}e^{2\pi i t}\varphi(t)\,dt\\
	&=e^{2\pi i x}\Phi_{\varphi}(0)+e^{2\pi ix}\big(\Phi_{\varphi}(2^{m}x)-\Phi_{\varphi}(0)\big),
\end{aligned}
\end{equation}
where $\Phi_{\varphi}(x)=\int_{-\infty}^x e^{2\pi i t}\varphi(t)\,dt$. Since $\varphi$ is non-degenerate (see Definition~\ref{nondeg}), we have 
$$
	g_{m}^{0}(0)=\Phi_\varphi(0)\neq 0
$$ 
provided $2^{m+1}\le 1/2$. We also note that $g_{m}^{0}$ are continuous functions.
We have
$$
		\big\|S^\varphi_{\I} f_0\big\|_{X_*} = \bigg|\Big\{2^{js}\big\|\big\{g_{m}^{j}\big\}_{m}\big\|_{L^{\infty}(l^2)}\Big\}_{j}\bigg|_{l^q}
		\ge \big\|\big\{g_{m}^{0}\big\}_m\big\|_{L^{\infty}(l^2)}=+\infty.
	$$

\subsection{Proof of Theorem~\ref{ThTrib}, part (iii)}
By definition~\eqref{FInfty} we have $\dot{F}_{\infty\infty}^s=\dot{B}_{\infty\infty}^s$, $s>0$, and it remains to prove the statement 
for $X = \dot{F}^0_{\infty 2}\cong\BMO$. We show that the same example as above gives an unbounded operator from $X$ to~$X_*$. 

Consider $\{g_m\}_m$ as an element of the quotient space~$X_*$. 
It is clear that $g_{m}(x)$ are bounded by $\|\varphi\|_{1}$ uniformly in $m$ and $x$. Therefore, for a sequence~$\{P_m\}_m$ of polynomials, the expression
$$
	\big\|\{g_{m}-P_m\}_m\big\|_{\BMO(l^2)}
$$ 
could be finite only if $\deg P_m\le 1$ for all~$m$. 

Next, suppose $2^{m+1}\le 1/2$ and $x \in [0,\, 1/2]$. Then by \eqref{eq:g}, we obtain
\[
|g_{m}(x)-P_m(x)-c_m|\ge|\Phi_\varphi(0)||\exp(2\pi ix)-a_mx-b_m|-|\Phi_\varphi(2^mx)-\Phi_\varphi(0)|.
\]
Then there exist subsets $E_m \subset [0,\,1/2]$ and a number $\gamma>0$ such that $|E_m|>\gamma$ and
\[
|g_{m}-P_m-c_m|>c_0|\Phi_\varphi(0)|\quad\mbox{on}\quad E_m
\]
provided $m$ is small enough.
Then
\[
\big\|\{g_{m}-P_m\}_m\big\|_{\BMO(l^2)}=\infty.\]

\end{document}